\documentclass[a4paper,11pt]{amsart}
\pdfoutput=1
\usepackage[T1]{fontenc}
\usepackage[utf8]{inputenc}
\usepackage{lmodern}
\usepackage[english]{babel}

\usepackage{microtype}
\usepackage[]{geometry}
\usepackage{amsmath,amssymb,mathrsfs}
\usepackage{mathtools}

\usepackage{subcaption}

\usepackage{bookmark}
\usepackage{hyperref}

\usepackage{color}
\usepackage{tikz}

\newtheorem{thm}{Theorem}[section]
\newtheorem{coro}[thm]{Corollary}
\newtheorem{lemma}[thm]{Lemma}
\newtheorem{prop}[thm]{Proposition}

\theoremstyle{definition}
\newtheorem{defn}[thm]{Definition}
\newtheorem{remark}[thm]{Remark}

\newcommand{\E}{\mathcal{E}}
\newcommand{\K}{\mathcal{K}}
\newcommand{\A}{\mathcal{A}_\omega^{A,B}}
\newcommand{\Am}{\mathcal{A}_{\omega,m}^{A,B}}
\newcommand{\M}{\mathcal{M}_\omega^{A,B}}
\newcommand{\Mm}{\mathcal{M}_{\omega,m}^{A,B}}
\newcommand{\F}{\mathcal{F}_\omega}
\newcommand{\Fm}{\mathcal{F}_{\omega,m}}
\newcommand{\Rt}{{\tilde{\mathbb{R}}}_m^N}
\newcommand{\uom}{u_\omega^{A,B}}
\newcommand{\uomm}{u_{\omega,m}^{A,B}}
\newcommand{\N}{\mathbb N}

\newcommand{\setR}{{\mathbb{R}^N}}
\newcommand{\eps}{\varepsilon}
\newcommand{\di}{\mathop{}\!\mathrm{d}}

\newcommand{\tR}{{\tilde{\mathbb{R}}}^N}

\let\epsilon\varepsilon

\DeclareMathOperator{\dist}{dist}

\usepackage[initials]{amsrefs}

\numberwithin{equation}{section}

\usepackage{url}

\DefineSimpleKey{bib}{nomecorto}

\begin{document}

\title[]{Minimizers for a fractional Allen-Cahn equation in a periodic medium}


\author{Dayana Pagliardini}

\address{
Scuola Normale Superiore,
Piazza dei Cavalieri 7,
56126 Pisa, Italy}
\email{dayana.pagliardini@sns.it}

\begin{abstract}
We aim to study the solutions of a fractional mesoscopic model of phase transitions in a periodic medium. After investigating the geometric properties of the interface of the associated minimal solutions, we construct minimal interfaces lying to a strip of prescribed direction and universal width.
\end{abstract}

\date{September 24, 2017}


\maketitle

\section{Introduction}\label{sec1}

For $N\ge 2$ we consider the energy functional

\begin{equation}\label{funct}
\mathcal{E}(u):=\frac{1}{2}\int_{\setR}\int_\setR |u(x)-u(y)|^2K(x,y)\di x\di y+\int_\setR W(x,u(x))\di x+\int_\setR H(x)u(x)\di x.
\end{equation}
The function $K:\setR \times \setR\rightarrow [0,+\infty]$ is measurable, symmetric and comparable to the kernel of the fractional laplacian, i.e.\
\begin{equation}\label{K1}
K(x,y)=K(y,x)\quad \text{ for a.e.}\; x,y\in \setR \tag{K1}
\end{equation}
and 
\begin{equation}\label{K2}
\frac{\lambda\chi_{(0,1)}(|x-y|)}{|x-y|^{N+2s}}\le K(x,y)\le \frac{\Lambda}{|x-y|^{N+2s}}\quad \text{ for a.e.}\; x,y\in \setR,\tag{K2}
\end{equation}
for some $\Lambda\ge\lambda>0$ and $s\in (0,1)$.

The function $H\in L^\infty(\mathbb{R}^N)$ is a small perturbation of the fractional Allen-Cahn functional. So we assume that
\begin{equation}\label{Hnormapiccola}
\sup_{\mathbb{R}^N}|H|\le \eta,\tag{H1}
\end{equation}
for $\eta$ sufficiently small, depending on $N$ and on the structural constants of the problem. We also assume that $H$ has zero-average and it is $\mathbb{Z}^N$-periodic, i.e.\
\begin{equation}\label{H0average}
\int_{[0,1]^N} H(x)\di x=0\tag{H2}
\end{equation}
and
\begin{equation}\label{Hperiodica}
H(x+k)=H(x)\qquad \forall\; k\in \mathbb{Z}^N.\tag{H3}
\end{equation}

The map $W:\setR \times \mathbb{R}\rightarrow [0,+\infty)$ is the standard double well potential, i.e.\ it is a bounded measurable function such that
\begin{equation}\label{W1}
W(x,\pm 1)=0\quad \text{ for a.e.}\; x\in \setR, \tag{W1}
\end{equation}
and for any $\theta\in [0,1)$
\begin{equation}\label{W2}
\inf_{\substack {x\in \setR\\ |r|\le \theta}}W(x,r)\ge\gamma(\theta)\tag{W2} 
\end{equation}
where $\gamma:[0,1)\rightarrow \mathbb{R}^+$ is a non-increasing function. We assume that $W$ is differentiable in the second component, with partial derivative locally bounded in $r\in \mathbb{R}$ and uniformly in $x\in \setR$, that is
\begin{equation}\label{W3}
W(x,r)|W_u(x,r)|\le W^*\quad \text{for a.e.}\; x\in \setR\text{ and any }r\in [-1,1] \tag{W3}
\end{equation}
for some $W^*>0$.
Moreover, since we want to model a periodic environment, we require both $K$ and $W$ to be periodic under integer translations:
\begin{equation}\label{K3}
K(x+k,y+k)=K(x,y)\quad \text{for a.e.}\; x,y\in \setR\text{ and any }k\in \mathbb{Z}^N\tag{K3}
\end{equation}
and 
\begin{equation}\label{W4}
W(x+k,r)=W(x,r)\quad \text{for a.e.}\; x\in \setR\text{ and any }k\in \mathbb{Z}^N, \tag{W4}
\end{equation}
for any fixed $r\in \mathbb{R}$.
Finally we require that 
\begin{equation}\label{W5}
W_u(x,-1-r)\le -c\quad \text{and}\quad W_u(x,1+r)\ge c \tag{W5}
\end{equation}
for any $r\ge \delta_0$ with $\delta_0 \in (0,1/10)$, and  suitable $c>0$, and
\begin{equation}\label{W6}
W(x,-1+r)=W(x,1+r)\tag{W6}
\end{equation}
for any $r\in[-\delta_0,\delta_0]$.

The functional $\eqref{funct}$ is composed by three terms (the first two give us the fractional Allen-Cahn equation):
\begin{itemize}
\item a ``kinetic interaction term'' $|u(x)-u(y)|^2K(x,y)$, which penalizes the phase changes of the system;
\item a double-well potential term $W$, which penalizes considerable deviations from the ``pure phase'' $\pm 1$;
\item a ``mesoscopic term'' $Hu$, which is ``neutral'' in the average and at each point it prefers one of the two phases.
\end{itemize}
Hence we have a model of phase coexistence (the ``pure phases'' respresented by $+1$ and $-1$) where $u:\setR \rightarrow [-1,+1]$ is a state parameter.

The fractional exponent $s\in (0,1)$ represents the fact that this model considers long-range particle interactions (and it can produce, depending on the value of $s$, local or non-local effect, see \cites{SV12, SV}).

The problem of plane-like minimizers, i.e.\ minimizers that stay at a finite distance from a plane, along every direction, is widely studied in recent years.

First of all we recall \cite{CDLL} where the authors considered an elliptic integrand $\mathfrak{I}$ in $\setR$ (but also functionals involving volume terms and suitable manifolds), periodic under integer translations, and they proved that for all plane in $\setR$ there exists at least one minimizer of $\mathfrak{I}$ with a bounded distance from this plane.

Some years later, in \cite{CV}, Cozzi and Valdinoci considered for $N \ge 2$ the nonlocal energy 
\begin{equation}\label{valdifunct}
\mathtt{E}(u):=\frac{1}{2}\int_{\setR}\int_\setR |u(x)-u(y)|^2K(x,y)\di x\di y+\int_\setR W(x,u(x))\di x,
\end{equation}
with $K$ and $W$ satisfying $\eqref{K1}$-$\eqref{K3}$ and $\eqref{W1}$-$\eqref{W6}$ respectively.

They constructed minimizers of this functional with interfaces in a slab of prescribed direction and bounded size (independently on the direction). 

The analogous result for $s=1$ was proved in \cite{V}, where the first addendum of $\mathtt{E}$ is replaced by 
\[
\int \langle A(x)\nabla u(x),\nabla u(x)\rangle \di x
\]
with $A$ bounded and uniformly elliptic matrix; some other generalizations were analyzed in \cites{PV, DLLV, BV}.

In particular, in \cite{NV}, Novaga and Valdinoci studied the energy 
\begin{equation}
E_\Omega(u):=\int_\Omega \Big( |\nabla u(x)|^2+W(x,u)+H(x)u(x)\Big)\di x,
\end{equation}
where $\Omega \subseteq \setR$ is a bounded domain, $u\in W^{1,2}(\Omega)$, $H$ and $W$ satisfy $\eqref{Hnormapiccola}$-$\eqref{H0average}$ and $\eqref{W1}$-$\eqref{W6}$ respectively.

They investigated geometric properties of the interfaces of the associated minimal solutions and they gave density estimates for the level sets. This allowed them to construct, in the periodic setting, minimal interfaces lying to a prescribed strip. These results are related with a PDE version of the Mather theory for dynamical systems, see \cites{M,Mo, B, CDLL, V, RS}.

In the same spirit of \cite{NV} and \cite{CV} we want to study functional $\eqref{funct}$, that is ``the fractional counterpart'' of $E_\Omega$ and is given by the functional $\mathtt{E}$ with the addition of the ``mesoscopic term''.

Our main goal is to construct minimal interfaces lying to a strip of universal size. Roughly speaking, given any vector $\omega \in \setR \setminus \{0 \}$ we look for minimizers having most of the transition between the pure states in a strip orthogonal to $\omega$ and of universal width:

\begin{thm}\label{Mainth}
Let $s\in (0,1)$ and $N\ge 2$. Suppose that the kernel $K$ and the potential $W$ satisfy $\eqref{K1}$-$\eqref{K3}$ and $\eqref{W1}$-$\eqref{W6}$ respectively.

Given $\theta \in (0,1)$, there exists a positive constant $M_0$ depending only on $\theta$ and on universal quantities, such that, for any $\omega \in \setR \setminus \{ 0\}$, there is a class $A$ minimizer $u_\omega$ of the functional $\E$ for which we have
\[ 
\{|u_\omega|<\theta \}\subset \Big \{ x\in \setR : \frac{\omega}{|\omega|}\cdot x\in [0,M_0]\Big \}.
\]
Moreover,
\begin{itemize}
\item if $\omega \in \mathbb{Q}^N \setminus \{0 \}$, $u_{\omega}$ is periodic with respect to $\sim_{\omega}$;
\item if $\omega \in \setR \setminus \mathbb{Q}^N$, $u_\omega$ is the uniform limit on compact subsets of $\setR$ of a sequence of periodic class $A$ minimizers.
\end{itemize}
\end{thm}

We prove Theorem \ref{Mainth} using geometric and variational tools introduced in \cite{CDLL} and \cite{V} and then adapted in \cite{CV} to deal with nonlocal interactions. Fixed $\omega \in \mathbb{Q}^N \setminus \{0\}$ we will consider the strip
\[
S_\omega^M:=\{x\in \mathbb{R}^N: \omega \cdot x \in [0,M]\},
\]
where $M>0$, and the quotient space $\tilde{\mathbb{R}}^N$ which allows us to gain compactness. This will be necessary to obtain a minimizer $u_\omega^M$ w.r.t.\ periodic perturbations with support in $S_\omega^M$. Thanks to geometrical arguments, if $M/|\omega|$ is larger than some universal parameter $M_0$, $u_\omega^M$ becomes a class $A$-minimizer for $\E$ (defined in Section \ref{sec2}). Since $M_0$ does not depend on the fixed direction $\omega$, we can pass to the limit on rational directions and deduce the result for an irrational vector $\omega \in \mathbb{R}^N \setminus \mathbb{Q}^N$.

We stress that the energy and density estimates is the standard technique to show that $u_\omega^M$ is a class $A$-minimizer. These estimates have been obtained in \cites{CC,SV} (in different settings), but their framework is different from ours. Thus we use the H\"{o}lderianity of local minimizers of $\E$ and an energy estimate.

Finally we point out that the addition of the term $Hu$ to $\eqref{valdifunct}$ changes the ``pure phases'' from $\pm 1$ into periodic functions, introducing a considerable difference with respect to \cite{CV}. Indeed, this fact produces a volume term in the energy that requires a renormalization as in  \cite{NV}.

The paper is organized as follows. In Section \ref{sec2} we explain the notation and we enounce our main theorem; in Secton \ref{sec3} we remind an important result about the regularity of minimizers and we prove an energy estimate. Section \ref{sec4} is devoted to investigate some geometric properties of minimizers (doubling property, Birkhoff property, etc.) and to prove our Theorem \ref{Mainth} under the additional assumption that $K$ has a fast decay at infinity. We will show our result both for rational vectors and, thanks to an approximation argument, for irrational vectors. Finally in Section \ref{sec5} we prove Theorem \ref{Mainth} for general kernels.

\section{Notation}\label{sec2}

In this section we introduce the framework that we will use throughout this paper.
\begin{defn}
Let $\Omega \subseteq \mathbb{R}^N$ be a bounded domain and $s\in (0,1)$. We define
\[
H^s(\Omega):=\Big \{ u\in L^2(\Omega): \frac{|u(x)-u(y)|}{|x-y|^{N/2+s}}\in L^2(\Omega \times \Omega)\Big \},
\]
i.e.\ an intermediary Banach space between $L^2(\Omega)$ and $H^1(\Omega)$, endowed with the natural norm
\[
\|u\|_{H^s(\Omega)}:=\Big (\int_\Omega |u|^2\di x+ \int_\Omega \int_\Omega \frac{|u(x)-u(y)|^2}{|x-y|^{N+2s}}\di x \di y \Big )^\frac{1}{2},
\]
where 
\[
[u]_{H^s(\Omega)}:=\Big(\int_\Omega \int_\Omega \frac{|u(x)-u(y)|^2}{|x-y|^{N+2s}}\di x \di y \Big )^\frac{1}{2} 
\]
is called Gagliardo (semi)-norm.
\end{defn}

\begin{defn}
Fixed $\omega \in \mathbb{Q}^N\setminus \{ 0\}$, we define in $\setR$ the relation
\[
x \sim_\omega y \Longleftrightarrow y-x=k\in \mathbb{Z}^N\; \text{with}\; \omega \cdot k=0.
\]
\end{defn}
It is easy to see that $\sim_\omega$ is an equivalence relation and we denote with
\[
\tilde{\mathbb{R}}_\omega^N:=\setR/_{\sim_\omega}
\]
the associated quotient space.

A function $u:\setR \rightarrow \mathbb{R}$ is said to be periodic with respect to $\sim_\omega$ if 
\[
u(x)=u(y)\quad \text{for any}\; x,y \in \setR\; \text{such that}\; x\sim_\omega y.
\]

When the context is clear, we will write $\sim$ and $\tilde{\mathbb{R}}^N$ to refer to $\sim_\omega$ and $\tilde{\mathbb{R}}_\omega^N$.

If we consider a set $\Omega\subseteq \setR$, we define the total energy $\E$ of $u:\setR \rightarrow \mathbb{R}$ in $\Omega$ as
\begin{equation}\label{FinOmega}
\E(u,\Omega):=\frac{1}{2}\int\int_{\mathcal{C}_\Omega} |u(x)-u(y)|^2K(x,y)\di x\di y+\int_\Omega W(x,u(x))\di x+\int_\Omega H(x)u(x)\di x,
\end{equation}
where 
\begin{equation}
\begin{aligned}
\mathcal{C}_\Omega&:=(\setR \times \setR)\setminus ((\setR\setminus \Omega)\times (\setR \setminus \Omega))\\
&=(\Omega \times \Omega)\cup (\Omega \times(\setR \setminus \Omega))\cup((\setR \setminus \Omega)\times \Omega).
\end{aligned}
\end{equation}
Observe that if $\Omega=\setR$ the energy $\eqref{FinOmega}$ coincides with $\eqref{funct}$.

Then, setting for all $U, V \subseteq \setR$
\[
\mathscr{K}(u;U;V):=\frac{1}{2}\int_U \int_V |u(x)-u(y)|^2K(x,y)\di x\di y,
\]
thanks to $\eqref{K1}$, we can see $\E(u,\Omega)$ as the sum of the kinetic part
\[
\K(u;\Omega;\Omega)+2\K(u;\Omega;\setR \setminus \Omega)
\]
and the potential part
\[
\mathscr{P}(u;\Omega):=\int_\Omega \Big(W(x,u(x))+H(x)u(x) \Big)\di x.
\]
Assuming from now on that every set and every function is measurable, we give the following
\begin{defn}\label{Def1}
Let $\Omega \subseteq \setR$ be a bounded set. A function $u$ is a local minimizer of $\E$ in~$\Omega$ if $\E(u,\Omega)<+\infty$ and 
\[
\E(u,\Omega)\le\E(v,\Omega)
\]
for any $v\equiv u$ in $\setR \setminus \Omega$.
\end{defn}
\begin{remark}\label{minimisottoins}[Remark 1.2 of \cite{CV}]
A minimizer $u$ of $\Omega$ is also a minimizer on every subset of $\Omega$.
\end{remark}
Since our aim is to construct functions with minimizing properties in $\setR$, we have to precise how we extend Definition \ref{Def1} to the full space.
\begin{defn}
A function $u$ is called a class $A$-minimizer of the functional $\E$ if it is a minimizer of $\E$ in $\Omega$ for any bounded set $\Omega \subseteq \setR$.
\end{defn}

\section{Regularity of the minimizers and energy estimate}\label{sec3}
In this section we want to prove that local minimizers of $\E$ are H\"{o}lder continuous functions with a growing energy inside large balls.

Let $\Omega\subseteq \setR$ be an open and bounded set, $s\in (0,1)$ and $K$ a measurable kernel that satisfies $\eqref{K1}$ and $\eqref{K2}$. If $u:\setR \rightarrow \mathbb{R}$ is a measurable function, we say that $u\in X(\Omega)$ if 
\[
u_{|_{\Omega}}\in L^2(\Omega)\quad \text{and}\quad (x,y)\mapsto (u(x)-u(y))\sqrt{K(x,y)}\in L^2(\mathcal{C}_\Omega).
\]
Then we denote with $X_0(\Omega)$ the subspace of $X(\Omega)$ given by functions vanishing a.e.\ outside $\Omega$. It is easy to see that by $\eqref{K2}$ it results $H^s(\setR)\subset X(\Omega)\subseteq H^s(\Omega)$ and if $\Omega'\subseteq \Omega$ we have $X_0(\Omega')\subseteq X_0(\Omega)\subset H^s(\setR)$.

Now we call 
\[
\mathcal{D}_K(u,\varphi)=\int_\setR \int_\setR (u(x)-u(y))(\varphi(x)-\varphi(y))K(x,y)\di x \di y
\]
observing that it is well-defined for example when $u\in X(\Omega)$ and $\varphi \in X_0(\Omega)$.

Let $f\in L^2(\Omega)$. We call $u\in X(\Omega)$ a supersolution of
\begin{equation}\label{Dirichletform}
\mathcal{D}_k(u,\cdot)=f\quad \text{in}\; \Omega
\end{equation}
if
\begin{equation}\label{supersol}
\mathcal{D}_k(u,\varphi)\ge\langle f,\varphi \rangle_{L^2(\setR)}\quad \text{for any non-negative}\; \varphi \in X_0(\Omega).
\end{equation}
Similarly, we say that $u\in X(\Omega)$ is a subsolution of $\eqref{Dirichletform}$ if
\begin{equation}\label{subersol}
\mathcal{D}_k(u,\varphi)\le\langle f,\varphi \rangle_{L^2(\setR)}\quad \text{for any non-negative}\; \varphi \in X_0(\Omega)
\end{equation}
and we tell that $u\in X(\Omega)$ is a solution of $\eqref{Dirichletform}$ if 
\begin{equation}\label{ssol}
\mathcal{D}_k(u,\varphi)=\langle f,\varphi \rangle_{L^2(\setR)}\quad \text{for any}\; \varphi \in X_0(\Omega).
\end{equation}
Obviously $u$ is a solution of $\eqref{Dirichletform}$ if it is a subsolution and a supersolution.

Thanks to these definitions we can show the regularity of the minimizers of $\E$.

\begin{thm}\label{holderianita}
Take $s_0\in (0,1/2)$ and let $s\in [s_0,1-s_0]$. If $u$ is a bounded local minimizer of $\E$ in a bounded open set $\Omega\subseteq \setR$, then $u\in C_{\text{loc}}^{0,\alpha}(\Omega)$ for some $\alpha \in (0,1)$. The exponent $\alpha$ only depends on $N$, $s_0$, $\lambda$ and $\Lambda$, while the $C^{0,\alpha}$ norm of $u$ on any $\Omega'\subset \subset \Omega$ may also depend on $\|u\|_{L^\infty(\setR)}$, $\|W_r(\cdot, u)\|_{L^\infty(\Omega)}$, $\eta$ and $\dist(\Omega',\partial \Omega)$.
\end{thm}

\begin{proof}
If we compute the first variation of $\eqref{FinOmega}$ we have that $u$ is a solution of the Euler-Lagrange equation $\eqref{Dirichletform}$ in $\Omega$ with $-f=W_r(\cdot,u)+H(\cdot)$. Since $\E(u,\Omega)<+\infty$ we have that $u\in X(\Omega)$. Moreover $u$, $H\in L^\infty(\setR)$ and $W_r$ locally bounded imply that $f$ is bounded in $\Omega$. So we can apply Theorem $2.1$ of \cite{CV} to obtain $C^{0,\alpha}$ regularity of $u$.
\end{proof}

Now we define
\begin{equation}\label{psidef}
\Psi_s(R):=\begin{cases}
R^{1-2s}\quad &\text{if}\; s\in (0,1/2)\\
\log R \quad &\text{if}\; s=1/2\\
1 \quad &\text{if}\; s\in (1/2,1)\\
\end{cases}
\end{equation}
and using a well-known result of \cite{SV} we want to show the energy estimate for minimizers:
\begin{thm}\label{stimedensita}
Let $N \in \N$, $s\in (0,1)$, $x_0\in \setR$ and $R\ge 3$. Suppose that $K$ and $W$ satisfy $\eqref{K1}$, $\eqref{K2}$ and $\eqref{W1}$, $\eqref{W3}$, respectively. If $u:\setR \rightarrow [-1,1]$ is a local minimizer of $\E$ in $B_{R+2}(x_0)$, then
\begin{equation}
\E(u,B_R(x_0))\le CR^{N-1}\Psi_s(R),
\end{equation}
for some constant $C>0$ which depends on $N$, $s$, $\Lambda$ and $W^*$.
\end{thm}

\begin{proof}
By \cite{CV}*{Proposition $3.1$} and \cite{NV}*{Lemma $2.1$} we have
\begin{equation}
\begin{aligned}
\E(u,B_{R}(x_0))&=\frac{1}{2}\int \int_{C_{B_{R}}}|u(x)-u(y)|^2K(x,y)\di x \di y+\int_{B_{R}}\Big(W(x,u(x))+H(x)u(x) \Big)\di x\\
&\le CR^{N-1}\Psi_s(R)+C\eta R^{N-1}\le CR^{N-1}\Psi_s(R),
\end{aligned}
\end{equation}
where $C$ denotes a positive constant depending at most on $N$, $s$, $\Lambda$ and $W^*$.
\end{proof}

We conclude this section giving an auxiliary result that will be very useful in the next Section \ref{sec4}.

\begin{lemma}\label{maxmin&funct}
Let $s\in (0,1)$, $U$, $V\subseteq \setR$ be measurable sets and $u$, $v\in H_{\text{loc}}^s(\setR)$. Then
\begin{equation}\label{Kemaxmin}
\mathscr{K}(\min \{u,v\};U;V)+\mathscr{K}(\max \{u,v\};U;V)\le \mathscr{K}(u;U;V)+\mathscr{K}(v;U;V),
\end{equation}
and
\begin{equation}\label{Pemaxmin}
\mathscr{P}(\min \{u,v\};U)+\mathscr{P}(\max \{u,v\};V)\le \mathscr{P}(u;U)+\mathscr{P}(v;V).
\end{equation}
\end{lemma}

\begin{proof}
The first identity is proved in \cite{CV}*{Lemma $3.2$}.

The other identity is obvious if $u=v$. So we assume that $\min \{u,v\}=u$ and $\max \{u,v\}=v$ (but in the other case it is analogous). Then
\[
\mathscr{P}(\min \{u,v\};U)=\int_U\Big(W(x,u(x))+H(x)u(x)\Big)\di x=\mathscr{P}(u;U)
\]
and
\[
\mathscr{P}(\max \{u,v\};V)=\int_V\Big(W(x,v(x))+H(x)v(x)\Big)\di x=\mathscr{P}(v;V)
\]
from which $\eqref{Pemaxmin}$ follows.
\end{proof}
\section{Proof of Theorem \ref{Mainth} for rapidly decaying kernels}\label{sec4}
In this section we want to prove Theorem \ref{Mainth} assuming the following hypothesis on $K$:
\begin{equation}\label{K4}
K(x,y)\le \frac{\Gamma}{|x-y|^{N+\beta}}\quad \text{for a.e.}
\; x,y \in \setR\; \text{such that}\; |x-y|\ge \bar{R}\; \text{with}\; \beta >1,\tag{K4}
\end{equation}
for some constant $\Gamma$, $\bar{R}>0$. This assumption is only technical and we will remove it in the next section. However a fast decay of the kernel $K$ at infinity due to $\beta>1$ ensures us that there exists a competitor with finite energy in the large then, since geometric estimates will not depend on the quantities in $\eqref{K4}$, we can use a limit procedure.

We start showing that the functional $\E$ has a minimizer among all periodic functions. 

Let $s\in (0,1)$, $Q:=[0,1]^{N}$ and define $Q$-periodic functions in $H_{\text{loc}}^s(\setR)$ as
\begin{equation}
H_{\text{per}}^s(Q)=\{u\in H_{\text{loc}}^s(\setR)\; \text{such that}\; u(x+e_j)=u(x)\; \text{for all}\; x \in \setR \}
\end{equation}
where $\{e_1,\cdots e_N \}$ is the standard Euclidean base of $\setR$. 

With this notation in hand, proceeding as in \cite{CV}*{Lemma $5.1$}, we have the following
\begin{thm}
Assume $K$ and $W$ as in Theorem \ref{Mainth}.
Then the functional $\E$ attains its minimum in $H_{\text{per}}^s(Q)$. Moreover if $u$ is a minimizer, it is continuous and
\begin{equation}\label{stimamin}
\Big| |u(x)|-1\Big|\le \delta_0
\end{equation}
for any $x\in Q$, as long as $\eta$ is small enough.
\end{thm}

\begin{proof}
Consider $\{u_n \}_{n\in \N}$ be a minimizing sequence. By $\eqref{H0average}$ we may suppose that
\begin{equation}\label{3.1}
\E(u_k,Q)\le \E (1,Q)=0.
\end{equation}
Then, from $\eqref{W5}$ we have
\begin{equation}\label{3.2}
\begin{aligned}
\min &\{W(x,1+s)-W(x,1+\delta_0), W(x, -1-s)-W(x,-1-\delta_0) \}\ge c(s-\delta_0)\\
&\ge |H(x)(\delta_0-s)|
\end{aligned}
\end{equation}
for any $s\ge \delta_0$ and
\[
W(x,r)+H(x)r\ge 0
\]
as long as $|r|\ge C_0$ with $C_0$ sufficiently large if $\eta$ is small enough. Accordingly, by $\eqref{3.1}$,
\begin{equation}\label{3.4}
\int_Q \int_\setR |u(x)-u(y)|^2K(x,y)\di x \di y \le \int_{Q \cap \{|u_k|\le C_0 \}} |H(x)u_k(x)|\di x\le C_0 \mathcal{L}(Q)\eta.
\end{equation}
Hence we define
\begin{equation}
u_k^*(x)=\begin{cases}
u_k(x)\quad &\text{if}\quad |u_k(x)|<1+\delta_0\\
1+\delta_0\quad &\text{if}\quad u_k(x)\ge 1+\delta_0\\
-1-\delta_0\quad &\text{if}\quad u_k(x)\le -1-\delta_0\\
\end{cases}
\end{equation}
and thanks to $\eqref{3.2}$ we get that $\E(u_k^*,Q)\le \E(u_k,Q)$. So, less than replacing $u_k$ with $u_k^*$ we may assume that
\begin{equation}\label{3.5}
|u_k|\le 1+\delta_0.
\end{equation}
By $\eqref{3.4}$, $\eqref{3.5}$ and the compact embedding of $H^s(Q)$ in $L^2(Q)$ we obtain that $u_k \rightarrow u$ in $L^2(Q)$ so $u_k \rightharpoonup u$ in $H^s(Q)$ and, up to subsequences, $u_k \rightarrow u$ a.e.\ Therefore $u\in H_{\text{per}}^s(Q)$ and
\[
\liminf_{k\rightarrow \infty}\int_Q \int_\setR |u_k(x)-u_k(y)|^2K(x,y)\di x \di y\ge \int_Q \int_\setR |u(x)-u(y)|^2K(x,y)\di x \di y.
\]
Then Fatou's Lemma gives us
\[
\inf_{H_{\text{per}}^s(Q)}\E(\cdot, Q)=\liminf_{k\rightarrow \infty}\E(u_k,Q)\ge \E(u,Q)
\]
i.e.\ $u$ is the desired minimizer.

From Theorem \ref{holderianita} we have that $u$ is continuous, so it remains to prove $\eqref{stimamin}$. To do this, we take $u\in H_{\text{per}}^s(Q)$ minimizer for $\E(\cdot,Q)$ and define
\begin{equation}
u^*(x):=\begin{cases}
u(x)\quad &\text{if}\quad|u(x)|<1+\delta_0\\
1+\delta_0\quad &\text{if}\quad u(x)\ge 1+\delta_0\\
-1-\delta_0\quad &\text{if}\quad u(x)\le -1-\delta_0.\\
\end{cases}
\end{equation}
By $\eqref{3.2}$ and since $u$ is a minimizer, we have
\[
0\le \E(u^*,Q)-\E(u,Q)\le -\frac{c}{2}\Big[ \int_{\{u>1+\delta_0\}}(u-1-\delta_0)+\int_{\{u<-1-\delta_0\}}(-u-1-\delta_0)\Big]\le 0,
\]
that is $|u|\le 1+\delta_0$. Then, if by contradiction
\[
-1+\delta_0\le u(x_0)\le 1-\delta_0\quad \text{for some}\; x_0\in Q,
\]
the uniform continuity of $u$ gives 
\[
-1+\frac{\delta}{2}\le u(x)\le 1-\frac{\delta_0}{2}
\]
for any $x\in B_\rho(x_0)$ for a suitable, universal $\rho>0$. As a consequence $W(x,u(x))\ge \text{const}$ for $x\in B_\rho(x_0)$, from which
\[
\E(u,Q)\ge \text{const}\cdot\mathcal{L}(B_\rho(x_0))-\eta \mathcal{L}(Q)>0=\E(1,Q)\ge \E(u,Q),
\]
that is a contradiction and proves $\eqref{stimamin}$.
\end{proof}
This theorem and $\eqref{W6}$ imply that the functional $\E(\cdot,Q)$ admits two minimizers $u_\pm \in~H_{\text{per}}^s(Q)$ such that $u_+=u_-+2$ and
\begin{equation}
\|u_\pm \mp 1\|_{L^\infty(Q)}=\delta_\eta < \delta_0.
\end{equation}

\begin{remark}
Observe that $\delta_\eta \rightarrow 0$ as $\eta \rightarrow 0$.
\end{remark}
Note that if $W(x, \cdot)$ is strictly convex in $[1-\delta_0,1+\delta_0]$ and $[-1-\delta_0,-1+\delta_0]$ these minimizers are the only global minimizers of $\E(\cdot,Q)$ in $H_{\text{per}}^s(Q)$ and from now on we assume that
\begin{equation}\label{ugualeenergia}
\E(u_+,Q)=\E(u_-,Q).
\end{equation}
\begin{remark}
Note that $\eqref{W6}$ (required for example by \cite{NV}) implies $\eqref{ugualeenergia}$.
\end{remark}

\subsection{Minimization with respect to periodic perturbations}\label{subsec4.1}
Given $\omega \in \mathbb{Q}^N \setminus \{ 0\}$ and $u:\setR \rightarrow \mathbb{R}$ a measurable function, we say that $u\in L^2(\tilde{\mathbb{R}}^N)$ if $u\in L_{\text{loc}}^2(\mathbb{R}^N)$ and $u$ is periodic with respect to $\sim$.

Hence, taken $A$, $B$ two real numbers such that $A<B$ and denoting with $\tilde{\mathbb{R}}^N$ any fundamental domain of the relation $\sim$, we define
\begin{equation}
\mathcal{A}_\omega^{A,B}:=\{u\in L_{\text{loc}}^2(\tilde{\mathbb{R}}^N):u(x)\ge 1-\delta_0\;\; \text{if}\;\; \omega \cdot x \le A\; \text{and}\; u(x)\le-1+\delta_0\;\; \text{if}\;\; \omega \cdot x \ge B\}
\end{equation}
the set of admissible functions and we consider
\begin{equation}\label{Fw}
\begin{aligned}
\mathcal{F}_\omega(u)&:=\frac{1}{2}\int_{\tR}\int_\setR \Big( |u(x)-u(y)|^2- |u_+(x)-u_+(y)|^2\Big)K(x,y)\di x \di y\\
&+\int_{\tR}\Big( W(x,u(x))-W(x,u_+(x))\Big)\di x+\int_{\tR}H(x)\Big(u(x)-u_+(x) \Big)\di x.
\end{aligned}
\end{equation}

We want to show that there exists an absolute minimizer of $\mathcal{F}_\omega$ in the class $\mathcal{A}_\omega^{A,B}$, i.e.\ there exists $u\in \mathcal{A}_\omega^{A,B}$ such that $\mathcal{F}_\omega(u)\le \mathcal{F}_\omega(v)$ for any $v\in \mathcal{A}_\omega^{A,B}$.

First of all we prove that $\mathcal{F}_\omega$ is not identically infinity on $\mathcal{A}_\omega^{A,B}$:
\begin{thm}\label{Fwisfinite}
Let $\bar{u}\in \A$ be defined as $\bar{u}(x):=\bar{\mu}(\omega \cdot x)$, where
\begin{equation}
\bar{\mu}(\omega \cdot x):=\begin{cases}
u_+\quad &\text{if}\quad \omega \cdot x\le A\\
u_+-\frac{(u_+-u_-)}{B-A}((\omega \cdot x)-A)\quad &\text{if}\quad A<\omega \cdot x\le B\\
u_-\quad &\text{if}\quad \omega \cdot x> B.\\
\end{cases}
\end{equation}
Then $\mathcal{F}_\omega(\bar{u})<+\infty$.
\end{thm}

\begin{proof}
Since the potential term of $\mathcal{F}_{\omega}$ vanishes  at $u_+$ and $u_-$ (thanks to $\eqref{ugualeenergia}$) for a.e.\ $x\in \setR$, it is obviously finite if we evaluate it in $\bar{u}$. So we only have to estimate the kinetic term and thanks to $\eqref{K2}$ and $\eqref{K4}$ it is sufficient to prove that
\begin{equation}\label{3.6}
\begin{aligned}
\int_{\tR}&\Big(\int_{B_{\bar{R}}(x)}\frac{|\bar{u}(x)-\bar{u}(y)|^2-|u_+(x)-u_+(y)|^2}{|x-y|^{N+2s}}\di y\\
&+\int_{\setR \setminus B_{\bar{R}}(x)}\frac{|\bar{u}(x)-\bar{u}(y)|^2-|u_+(x)-u_+(y)|^2}{|x-y|^{N+\beta}}\di y\Big)\di x<+\infty.
\end{aligned}
\end{equation}
Less than an affine transformation we may assume $\omega=e_N$ and for simplicity we may also suppose that $A=0$ and $B=1$ so that $\tR=[0,1]^{N-1}\times \mathbb{R}$. Accordingly $\eqref{3.6}$ is equivalent to
\begin{equation}\label{I}
I:=\int_{[0,1]^{N-1}\times \mathbb{R}}\int_{B_{\bar{R}}(x)}\frac{|\bar{u}(x)-\bar{u}(y)|^2-|u_+(x)-u_+(y)|^2}{|x-y|^{N+2s}}\di y \di x<+\infty
\end{equation}
and 
\begin{equation}\label{J}
J:=\int_{[0,1]^{N-1}\times \mathbb{R}}\int_{\setR \setminus B_{\bar{R}}(x)}\frac{|\bar{u}(x)-\bar{u}(y)|^2-|u_+(x)-u_+(y)|^2}{|x-y|^{N+\beta}}\di y\di x<+\infty.
\end{equation}
Recalling the definition of $\bar{u}$ it follows that
\begin{equation}\label{I1}
I=\int_{[0,1]^{N-1}\times [-\bar{R},\bar{R}+1]}\int_{B_{\bar{R}}(x)}\frac{|\bar{u}(x)-\bar{u}(y)|^2-|u_+(x)-u_+(y)|^2}{|x-y|^{N+2s}}\di y \di x
\end{equation}
and being $\bar{u}$ Lipschitz, we get
\begin{equation}
\begin{aligned}
I&\le 4(1+\delta_0)^2 \int_{[0,1]^{N-1}\times [-\bar{R},\bar{R}+1]}\Big(\int_{B_{\bar{R}}(x)}\frac{\di y}{|x-y|^{N+2s-2}}\Big)\di x\\
&=\frac{2N\alpha_N(1+\delta_0)^2}{1-s}(2\bar{R}+1)\bar{R}^{2-2s}
\end{aligned}
\end{equation}
which implies $\eqref{I}$.

Now to prove $\eqref{J}$ we write $J=J_1+J_2+J_3$ with
\[
J_1:=\int_{[0,1]^{N-1}\times [2,+\infty)}\Big(\int_{\setR \setminus B_{\bar{R}}(x)}\frac{|\bar{u}(x)-\bar{u}(y)|^2-|u_+(x)-u_+(y)|^2}{|x-y|^{N+\beta}}\di y\Big)\di x,
\]
\[
J_2:=\int_{[0,1]^{N-1}\times (-\infty,-1)}\Big(\int_{\setR \setminus B_{\bar{R}}(x)}\frac{|\bar{u}(x)-\bar{u}(y)|^2-|u_+(x)-u_+(y)|^2}{|x-y|^{N+\beta}}\di y\Big)\di x,
\]
\[
J_3:=\int_{[0,1]^{N-1}\times [-1,2]}\Big(\int_{\setR \setminus B_{\bar{R}}(x)}\frac{|\bar{u}(x)-\bar{u}(y)|^2-|u_+(x)-u_+(y)|^2}{|x-y|^{N+\beta}}\di y\Big)\di x.
\]
By the definition of $\bar{u}$ we have that
\begin{equation}
\begin{aligned}
J_1&\le \int_{[0,1]^{N-1}\times [2,+\infty)}\Big(\int_{\mathbb{R}^{N-1}\times (-\infty,1]}\frac{|u_-(x)-\bar{u}(y)|^2}{|x-y|^{N+\beta}}\di y\Big)\di x\\
&\le 4(1+\delta_0)^2\int_{[0,1]^{N-1}\times [2,+\infty)}\Big(\int_{\mathbb{R}^{N-1}\times (-\infty,1]}\frac{\di y}{|x-y|^{N+\beta}}\Big)\di x.
\end{aligned}
\end{equation}
Writing $x=(x',x_N)\in \mathbb{R}^{N-1}\times \mathbb{R}$, $y=(y',y_N)\in \mathbb{R}^{N-1}\times \mathbb{R}$ and substituing $z':=~(y'-x')/|x_N-y_N|$, we get
\begin{equation}
\begin{aligned}
\int_{\mathbb{R}^{N-1}\times (-\infty,1]}\frac{\di y}{|x-y|^{N+\beta}}&=\int_{-\infty}^1 |x_N-y_N|^{-N-\beta}\Big[\int_{\mathbb{R}^{N-1}}\Big(1+\frac{|x'-y'|^2}{|x_N-y_N|^2}\Big)^{-\frac{N+\beta}{2}}\di y' \Big]\di y_N\\
&=\int_{-\infty}^1 |x_N-y_N|^{-1-\beta}\Big[\int_{\mathbb{R}^{N-1}}(1+|z'|^2)^{-\frac{N+\beta}{2}}\di z' \Big]\di y_N\\
=&\frac{\Theta}{\beta}(x_N-1)^{-\beta},
\end{aligned}
\end{equation}
where
\[
\Theta:=\int_{\mathbb{R}^{N-1}}(1+|z'|^2)^{-\frac{N+\beta}{2}}\di z'<+\infty.
\]
Therefore
\[
J_1\le \frac{4(1+\delta_0)^2}{\beta}\Theta \int_2^{+\infty}(x_N-1)^{-\beta}\di x_N=4(1+\delta_0)^2\frac{\Theta}{(\beta-1)\beta},
\]
since $\beta >1$. Analogously it is easy to see that $J_2$ is finite too. Thus we pass to estimate $J_3$.

Since $\bar{u}$ is a bounded function we have
\begin{equation}
\begin{aligned}
J_3&\le 4(1+\delta_0)^2\int_{[0,1]^{N-1}\times [-1,2]}\Big(\int_{\setR \setminus B_{\bar{R}}(x)}\frac{\di y}{|x-y|^{N+\beta}}\Big)\di x=\frac{12 N \alpha_N}{\beta}\bar{R}^{-\beta}
\end{aligned}
\end{equation}
and $\eqref{J}$ follows.
\end{proof}
Note that condition $\eqref{K4}$ allows us to have the integrability of the first addendum of $\F$.

With this result in hand we can prove that
\begin{thm}\label{esistenza}
There exists an absolute minimizer of the functional $\mathcal{F}_\omega$ in the class $\mathcal{A}_\omega^{A,B}$.
\end{thm}

\begin{proof}
We use the standard Direct Method of the Calculus of variations.

By Theorem \ref{Fwisfinite} and since $u^+$ is minimizer for $\E$, we have that $\mathcal{F}_\omega\ge 0$ and hence
\[
m:=\inf \{\mathcal{F}_\omega(u): u\in \mathcal{A}_\omega^{A,B} \}\in [0,+\infty).
\]
So, if $\{u_j \}_{j\in \N}\subseteq \mathcal{A}_\omega^{A,B}$ is a minimizing sequence, we may suppose that
\begin{equation}\label{minimizzantestima}
|u_j|\le 1+\delta_0\quad \text{a.e.\ in}\; \setR.
\end{equation}
Then we consider an integer $k>\max \{-A,B \}$ and the Lipschitz domains
\[
\Omega_k:=\tR\cap \{x\in \setR: |\omega \cdot x|\le k \}.
\]
Thanks to $\eqref{minimizzantestima}$ and $\eqref{K2}$ we obtain
\begin{equation}\label{stimanormagagliardo}
\begin{aligned}
[u_j]_{H^s(\Omega_k)}^2&\le \int_{\Omega_k}\Big(\int_{B_1(x)} \frac{|u_j(x)-u_j(y)|^2}{|x-y|^{N+2s}}\di y \Big)\di x\\
&+4(1+\delta_0)^2\int_{\Omega_k}\Big(\int_{\setR \setminus B_1(x)} \frac{\di y}{|x-y|^{N+2s}} \Big)\di x\\
&\le\frac{2}{\lambda}\mathcal{F}_\omega(u_j,\Omega_k)+\frac{2}{\lambda}\int_{\Omega_k}\int_{\mathbb{R}^N}\frac{|u_+(x)-u_+(y)|^2}{|x-y|^{N+2s}}\di x \di y+\frac{2}{\lambda}\int_{\Omega_k}W(x,u_+(x))\di x\\
&+\frac{2}{\lambda}\int_{\Omega_k}H(x)u_+(x)\di x-\frac{2}{\lambda}\int_{\Omega_k}W(x,u_j(x))\di x-\frac{2}{\lambda}\int_{\Omega_k}H(x)u_j(x)\di x\\
&+2\frac{(1+\delta_0)^2}{s}N\alpha_N|\Omega_k|,
\end{aligned}
\end{equation}
where we denote with
\begin{equation}
\begin{aligned}
\mathcal{F}_\omega(u,\Omega_k)&:=\frac{1}{2}\int_{\Omega_k}\int_\setR \Big( |u(x)-u(y)|^2- |u_+(x)-u_+(y)|^2\Big)K(x,y)\di x \di y\\
&+\int_{\Omega_k}\Big( W(x,u(x))-W(x,u_+(x))\Big)\di x+\int_{\Omega_k}H(x)\Big(u(x)-u_+(x) \Big)\di x.
\end{aligned}
\end{equation}
Now we take $k\in \mathbb{N}$ such that $k\omega \in \mathbb{Z}^N$, so that $\Omega_k$ is a periodicity domain for $u_+$.
From this and the fact that $u_+$ is minimizer for $\E$ on all the domains $\Omega_k$, we get
\[
0\le \F(u_j,\Omega_k)\le \F(u_j,\tilde{\mathbb{R}}^N),
\]
so $\eqref{stimanormagagliardo}$ becomes
\begin{equation}
\begin{aligned}
[u_j]_{H^s(\Omega_k)}^2&\le\frac{2}{\lambda}\F(u_j)+ \frac{2}{\lambda}\int_{\Omega_k}\int_{\mathbb{R}^N}\frac{|u_+(x)-u_+(y)|^2}{|x-y|^{N+2s}}\di x \di y+\frac{2}{\lambda}\int_{\Omega_k}W(x,u_+(x))\di x\\
&+\frac{2}{\lambda}\int_{\Omega_k}H(x)u_+(x)\di x+2|\Omega_k|\Big(\frac{1}{\lambda}\eta +\frac{(1+\delta_0)^2}{s}N\alpha_N\Big).
\end{aligned}
\end{equation}
Hence $\{u_j\}_{j\in \N}$ is bounded in $H^s(\Omega_k)$ uniformly in $j$. Since $H^s(\Omega_k)\hookrightarrow \hookrightarrow L^2(\Omega_k)$ (see \cite{DPV}*{Theorem $7.1$}), less than extract a subsequence, $u_j\rightarrow u$ in $L^2(\Omega_k)$ and a.e.\ in $\Omega_k$. Now we use a diagonal argument (on $j$ and $k$) to find a subsequence $\{u_j^* \}_{j\in \N}$ of $\{u_j\}_{j\in \N}$ such that ${{u_j}^*} \rightarrow u$ a.e.\ in $\tR$. We may identify the $u_j^*$'s and $u$ with their $\sim$-periodic extension to $\setR$ so that the convergence will be in the full space $\setR$.

As a consequence $u\in \mathcal{A}_\omega^{A,B}$ and using Fatou's Lemma we get $\mathcal{F}_\omega(u)=m$ that concludes the proof.
\end{proof}

\subsection{The minimal minimizer}\label{minmin}
Define
\[
\mathcal{M}_\omega^{A,B}:=\{u\in  \mathcal{A}_\omega^{A,B}:\mathcal{F}_\omega(u)\le \mathcal{F}_\omega(v)\; \text{ for any}\; v\in \mathcal{A}_\omega^{A,B} \}
\]
the set of the absolute minimizers of $\mathcal{F}_\omega$ in $\mathcal{A}_\omega^{A,B}$. Observe that from Proposition \ref{esistenza}, $\mathcal{M}_\omega^{A,B}$ is not empty, hence we can introduce the following
\begin{defn}\label{minimalminimizer}
We call $u_\omega^{A,B}$ a minimal minimizer when it is the infimum of $\mathcal{M}_\omega^{A,B}$ if we consider $\mathcal{M}_\omega^{A,B}$ subset of the partially ordered set $(\mathcal{A}_\omega^{A,B}, \le)$. In particular $u_\omega^{A,B}$ is the unique function of $\mathcal{A}_\omega^{A,B}$ such that
\begin{equation}\label{1minmimnimizer}
u_\omega^{A,B}\le u \; \text{in}\; \setR \text{ for every}\; u\in \mathcal{M}_\omega^{A,B}
\end{equation}
and
\begin{equation}\label{2minminimizer}
\text{if}\; v\in \A \; \text{is such that}\; v\le u \; \text{in}\; \setR\ \text{for every}\; u\in \M,\; \text{then}\; v\le \uom\; \text{in}\; \setR.
\end{equation}
\end{defn}
The existence of $\uom$ is not obvious, so we will dedicate the rest of the section to show it.

First of all we need to prove that the minimum between two elements of $\M$ still belongs to $\M$. To do this we show the following

\begin{lemma}\label{mininM}
Let $A$, $A'$, $B$, $B'$ be real numbers such that $A<A'$ and $B<B'$ with $A<B$ and $A'<B'$. If $u \in \M$ and $v\in \mathcal{M}_\omega^{A',B'}$, then $\min \{u,v\}\in \M$.
\end{lemma}

\begin{proof}
Observing that $\min \{u,v\}\in \A$ and $\max \{u,v\}\in \mathcal{A}_\omega^{A',B'}$ and using Lemma \ref{maxmin&funct}, we get
\[
\F(\min \{u,v\})+\F(\max \{u,v\})\le \F(u)+\F(v).
\]
Now, since $v\in \mathcal{M}_\omega^{A',B'}$ we have
\[
\F(\min \{u,v\})+\F(\max \{u,v\})\le \F(u)+\F(\max \{u,v\})
\]
and hence
\[
\F(\min \{u,v\})\le \F(u),
\]
that is $\min \{u,v\}\in \M$.
\end{proof}
As a consequence, if we choose $A=A'$ and $B=B'$ we obtain this 
\begin{coro}\label{2elinMmininM}
If $u$, $v\in \M$, then $\min \{u,v\}\in \M$.
\end{coro}
At this point we can show that $\M$ is also closed with respect to take the minimum among a countable family of its elements:
\begin{lemma}\label{countablemin}
If $\{u_n \}_{n\in \N}$ is a sequence of elements in $\M$, then $\inf_{n\in \N}u_n \in \M$.
\end{lemma}

\begin{proof}
Define $u_*:=\inf_{j\in \N} u_j$ and inductively the sequence
\begin{equation}
v_j:=\begin{cases}
u_1\quad &\text{if}\; j=1\\
\min \{v_{j-1},u_j \}&\text{if}\; j\ge 2.
\end{cases}
\end{equation}
Corollary \ref{2elinMmininM} gives us that $\{v_j\}_{j\in \N}\subseteq \M$. On the other hand $v_j \rightarrow u_*$ a.e.\ in $\setR$, so from an application of Fatou's Lemma we have that $u_*\in \A$ and
\[
\F(u_*)\le \lim_{j\rightarrow +\infty} \F (v_j)=\F(v_k)
\]
for any $k\in \N$. Hence $u_* \in \M$.
\end{proof}

These results allow us to prove this

\begin{prop}
The minimal minimizer $\uom$ exists and belongs to $\M$.
\end{prop}

\begin{proof}
Since $\M$ is separable with respect to convergence a.e.\ (see \cite{CV}*{Proposition $B.2$}), for all $u\in \M$ we can find a sequence $\{u_n \}_{n\in \N}\subseteq \M$ from which we can extract a subsequence $\{u_{n_k}\}_{k \in \N}$ such that $u_{n_k}\rightarrow u$ a.e.\ in $\setR$.
We define
\[
\uom:=\inf_{n\in \N}u_n 
\]
and from Lemma \ref{countablemin} we get $\uom \in \M$.

We claim that $\uom$ is the minimal minimizer, that is we have to check $\eqref{1minmimnimizer}$ and $\eqref{2minminimizer}$.

Let $u\in \M$ and $\{u_{n_k} \}_{k\in \N}$ a subsequence of $\{u_n\}_{n\in \N}$ such that $u_{n_k}\rightarrow u$ a.e.\ in $\setR$. By definition $\uom \le u_{n_k}$ in $\setR$ for all $k\in \N$. Therefore, passing to the limit as $k\rightarrow +\infty$, we obtain $\eqref{1minmimnimizer}$.

In order to prove $\eqref{2minminimizer}$ we have to suppose the existence of $v\in \A$ such that $v\le u$ for all $u\in \M$. This implies $v\le u_n$ for all $n\in \N$. Hence $v\le \uom $ and $\eqref{2minminimizer}$ is proved.
\end{proof}
\subsection{The doubling property}\label{subsection4.3}
The doubling property, or no-symmetry breaking property, is an important feature of the minimal minimizer. In this subsection we want to show that $\uom$ is not only the minimal minimizer of $\M$, but also the minimal minimizer over the functions with periodicity multiple of $\sim$. To do this we introduce a few more notation.

We denote with $z_1,\cdots, z_{N-1}\in \mathbb{Z}^N$ some vectors spanning the $(N-1)$-dimensional lattice induced by $\sim$. If $k\in \mathbb{Z}^N$ is such that $\omega \cdot k=0$ we can write
\[
k=\sum_{i=1}^{N-1}\mu_i z_i,
\]
with $\mu_1, \cdots, \mu_{N-1}\in \mathbb{Z}$. Then we take $m\in \mathbb{N}^{N-1}$ and we define the equivalence relation $\sim_m$ as
\[
x\sim_m y \Leftrightarrow x-y=\sum_{i=1}^{N-1}\mu_im_i z_i\quad \text{for}\; \mu_1,\cdots, \mu_{N-1}\in \mathbb{Z}.
\]
We denote with $\Rt:=\setR/_{\sim_m}$ and with $L_{\text{loc}}^2(\Rt)$ the periodic functions of $L_{\text{loc}}^2(\tilde{\mathbb{R}}^N)$. Note that in $\Rt$ there are $m_1\cdots m_{N-1}$ copies of $\tR$ because the relation $\sim$ is stronger than $\sim_m$ and $L_{\text{loc}}^2(\tilde{\mathbb{R}}^N)\subseteq L_{\text{loc}}^2(\Rt)$.

We define the space
\[
\mathcal{A}_{\omega,m}^{A,B}:=\{u\in L_{\text{loc}}^2(\Rt):u(x)\ge 1-\delta_0\; \;\text{if}\;\; \omega \cdot x \le A\; \text{and}\; u(x)\le-1+\delta_0\;\; \text{if}\;\; \omega \cdot x \ge B\},
\]
i.e.\ the admissible functions related to the new equivalence relation. Then we consider the functional
\begin{equation}\label{Fwm}
\begin{aligned}
\mathcal{F}_{\omega,m}(u)&:=\frac{1}{2}\int_{\Rt}\int_\setR \Big( |u(x)-u(y)|^2- |u_+(x)-u_+(y)|^2\Big)K(x,y)\di x \di y\\
&+\int_{\Rt}\Big( W(x,u(x))-W(x,u_+(x))\Big)\di x+\int_{\Rt}H(x)\Big(u(x)-u_+(x) \Big)\di x
\end{aligned}
\end{equation}
and the set of absolute minimizers 
\[
\Mm:=\{u\in  \Am:\mathcal{F}_{\omega,m}(u)\le \mathcal{F}_{\omega,m}(v)\; \text{ for any}\; v\in \Am \}.
\]
We call $\uomm$ the minimal minimizer of $\Mm$ whose existence is assured by the same arguments of Subsection \ref{minmin}.

Finally we indicate the translation of a function $u:\setR \rightarrow \mathbb{R}$ in the direction $z\in \setR$ as
\begin{equation}
\tau_z u(x):=u(x-z)\quad \text{for any}\; x\in \setR.
\end{equation}
At this point we can show that the minimal minimizer in a class of larger period coincides to that in a class of smaller period:
\begin{prop}\label{minimiuguali}
For any $m\in \mathbb{N}^{N-1}$, it results $\uomm=\uom$.
\end{prop}

\begin{proof}
Without loss of generality we consider $m_1=2$ and $m_i=1$ for every $i=2, \cdots, N-~1$. (The general case is analogous but the notation is much heavier).

First we show that $\uom\in \Mm$, since this implies that $\uomm \le \uomm$. To do this we consider $\tau_{z_1} \uomm$ (i.e.\ the translation of $\uom$ in the doubled direction of $z_1$) and we observe that it is an element of $\Mm$. Defining
\begin{equation}
\hat{u}_{\omega,m}^{A,B}:=\min \{\uomm, \tau_{z_1} \uomm \},
\end{equation}
we may see that it is $\sim$-periodic, so $\hat{u}_{\omega,m}^{A,B}\in \A$. Then, using Lemma \ref{maxmin&funct} and arguing as in the proof of Lemma \ref{mininM}, we have
\begin{equation*}
\Fm(\uom)=2\F(\uom)\le 2\F(\hat{u}_{\omega,m}^{A,B})=\Fm(\hat{u}_{\omega,m}^{A,B})\le \Fm(\uomm).
\end{equation*}
As a consequence $\uom \in \Mm$ and so $\uomm \le \uom$, being $\uomm$ the minimal minimizer of $\Mm$.

On the other hand, since $\hat{u}_{\omega,m}^{A,B}\in \Mm$ and $\uom \in \Am$, we get
\begin{equation*}
\F(\hat{u}_{\omega,m}^{A,B})=\frac{1}{2}\Fm(\hat{u}_{\omega,m}^{A,B})\le \frac{1}{2}\Fm(\uom)=\F(\uom),
\end{equation*}
from which it follows that $\hat{u}_{\omega,m}^{A,B} \in \M$. Hence
\[
\uom \le \hat{u}_{\omega,m}^{A,B}\le \uomm
\]
and the proof is complete.
\end{proof}

\subsection{Minimization with respect to compact perturbations}

In this subsection we want to construct a class $A$-minimizer for $\E$, so we have to prove that the elements of $\M$ are also minimizers of the energy $\E$ with respect to compact perturbations in the strip
\[
S_\omega^{A,B}:=\{x \in \setR: \omega \cdot x\in [A,B] \}.
\] 
We call
\[
\tilde{S}_\omega^{A,B}:=S_\omega^{A,B}/{\sim_m}
\]
the quotient of the strip with respect to the relation $\sim_m$ and we show a relation between $\E$ and $\Fm$.
\begin{lemma}
Let $u\in \Am$ such that $\Fm(u)<\infty$. For any $\Omega \subset \subset \tilde{S}_{\omega,m}^{A,B}$, consider $v$ another bounded function such that $u=v$ in $\setR \setminus \Omega$ and denote with $\varphi:=v-u$. Calling $\tilde{v}$ and $\tilde{\varphi}$ the $\sim_m$-periodic extension to $\setR$ of $v_{|_{\Rt}}$ and $\varphi_{|_{\Rt}}$ respectively, we have
\begin{equation}\label{linkEFm1}
\E(v,\Rt)-\E(u,\Rt)=\Fm(\tilde{v})-\Fm(u)+\int_{\Rt} \int_{\setR \setminus \Rt}\tilde{\varphi}(x)\tilde{\varphi}(y)K(x,y)\di x \di y.
\end{equation}
In particular if $u\in \Mm$,
\begin{equation}\label{linkEFm2}
\E(v,\Rt)-\E(u,\Rt)\ge \int_{\Rt} \int_{\setR \setminus \Rt}\tilde{\varphi}(x)\tilde{\varphi}(y)K(x,y)\di x \di y.
\end{equation}
\end{lemma}

Observe that, being $\varphi$ compactly supported in $\tilde{S}_{\omega,m}^{A,B}$ and bounded, the right hand sides of  $\eqref{linkEFm1}$ and $\eqref{linkEFm2}$ are finite (see \cite{CV}*{Lemma $A.2$})

\begin{proof}
We prove the lemma in the case $m=(1, \cdots,1)$ but the general case is analogous, moreover we show only $\eqref{linkEFm1}$ because then $\eqref{linkEFm2}$ follows noticing that $\tilde{v}\in \Am$. Recalling the expression of $\E$ (see $\eqref{funct}$), we start by computing $\mathcal{K}(v, \tR, \setR \setminus \tR)$.

Proceeding as in \cite{CV}*{Lemma $4.4.1$}, we get
\begin{equation}
\begin{aligned}
\mathscr{K}(v, \tR, \setR \setminus \tR)&=\mathscr{K}(\tilde{v}, \tR, \setR \setminus \tR)+\mathscr{K}(u, \tR, \setR \setminus \tR)\\
&-\mathscr{K}(v, \setR \setminus \tR, \tR)+\int_{\tR} \int_{\setR \setminus \tR}\tilde{\varphi}(x)\tilde{\varphi}(y)\mathscr{K}(x,y)\di x \di y.
\end{aligned}
\end{equation}
Then we note that 
\[
\mathscr{K}(v, \tR, \tR)=\mathscr{K}(\tilde{v}, \tR, \tR)\quad \text{and}\quad \mathscr{P}(v, \tR)=\mathscr{P}(\tilde{v}, \tR)
\]
and recalling the definitions of $\E$ and $\F$ we conclude the proof.
\end{proof}
Now we are ready to prove that the absolute minimizers of $
\Fm$ in $\Am$ are also minimizers for $\E$ with respect to compact perturbations in $\tilde{S}_{\omega,m}^{A,B}$:

\begin{prop}\label{minimizerisminE}
If $u\in\Mm$, then it is a local minimizer of $\E$ in every open set $\Omega \subset \subset \tilde{S}_{\omega,m}^{A,B}$, i.e.\
\begin{equation}\label{minE}
\E(u,\Omega)\le \E(v,\Omega)
\end{equation}
for all $v\equiv u$ in $\setR \setminus \Omega$.
\end{prop}

\begin{proof}
Without loss of generality we may suppose that $\E(v,\Omega)<+\infty$ and $|v|\le~1+\delta_0$ a.e.\ in $\setR$. Let $\varphi:=v-u$ and note that spt $\varphi \subset \Omega$. We claim that $\eqref{minE}$ holds with $\Omega$ replaced by $\Rt$, i.e.\
\begin{equation}\label{minE2}
\E(u,\Rt)\le \E(v,\Rt).
\end{equation}
Then Remark \ref{minimisottoins} will implies $\eqref{minE}$. 

To show $\eqref{minE2}$ we observe that if $\varphi$ is either non-negative or non-positive, then $\eqref{minE2}$ is a direct consequence of $\eqref{linkEFm2}$. Moreover, if $\varphi$ is sign-changing, we consider $\min \{u,u+\varphi \}$ and $\max \{u,u+\varphi \}$. From Lemma \ref{maxmin&funct} we get 
\[
\E(\min \{u,u+\varphi \},\Rt)+\E(\max \{u,u+\varphi \},\Rt)\le \E(u, \Rt)+\E(u+\varphi, \Rt).
\]
Moreover, noticing that
\[
\min \{u,u+\varphi \}=u-\varphi_-\quad \text{and}\quad \max \{u,u+\varphi \}=u+\varphi_+
\]
and using $\eqref{linkEFm2}$, we obtain
\begin{equation}
\begin{aligned}
2\E(u, \Rt)&\le \E(u-\varphi_-,\Rt)+\E(u+\varphi_+,\Rt)=\E(\min \{u,u+\varphi \},\Rt)\\
&+\E(\max \{u,u+\varphi \},\Rt)\le \E(u, \Rt)+\E(u+\varphi,\Rt)
\end{aligned}
\end{equation}
that is our thesis.
\end{proof}
As a consequence of this proposition and Subsection \ref{subsection4.3}, we have the following
\begin{coro}\label{uomislocmin}
The minimal minimizer $\uom$ is a local minimizer of $\E$ for every $\Omega \subset \subset S_\omega^{A,B}$.
\end{coro}

\begin{proof}
Given $\Omega$, consider $m\in \mathbb{N}^{N-1}$ such that $\Omega \subset \subset \tilde{S}_{\omega,m}^{A,B}$. Thanks to Proposition \ref{minimiuguali}, $\uom$ is the minimal minimizer with respect to $\Mm$ and Proposition \ref{minimizerisminE} implies that $\uom$ is a local minimizer of $\E$ in $\Omega$.
\end{proof}

\subsection{The Birkhoff property}
In this subsection we recall a geometric property of the level sets of the minimal minimizer called the Birkhoff property, or non-self intersection property, representing the fact that the level sets of the minimal minimizers are ordered under translations.

We start giving some useful notation. We define
\[
\tau_z E:= E+z=\{x+z : x\in E \}
\]
the translation of a set $E\subseteq \setR$ with respect to $z\in \setR$ and observe that for a sublevel set (and analogously for a superlevel set)
\[
\tau_z \{u<\theta \}=\{\tau_z u<\theta \}.
\]
\begin{defn}\label{birkhoff}
We say that $E \subseteq \setR$ has the Birkhoff property with respect to a vector $\bar{\omega}\in \setR$ if
\begin{itemize}
\item $\tau_k E \subseteq E$ for any $k\in \mathbb{Z}^N$ such that $\bar{\omega}\cdot k \le 0$, and
\item 
$\tau_k E \supseteq E$ for any $k\in \mathbb{Z}^N$ such that $\bar{\omega}\cdot k \ge 0$.
\end{itemize}
\end{defn}

We call Birkhoff set a set satisfying the Birkhoff property and we recall an useful result on Birkhoff sets (see \cite{CV}*{Proposition $4.5.2$}):
\begin{prop}\label{Birkricito}
Let $E \subseteq \setR$ a Birkhoff set with respect to $\bar{\omega}\in \setR \setminus \{0 \}$ and containing a ball $B_{\sqrt{N}}$ of radius $\sqrt{N}$.
Then $E$ contains a half-space including the center of the ball, is delimited by a hyperplane orthogonal to the vector $\bar{\omega}$ and is such that $\bar{\omega}$ points outside of it.
\end{prop}
Now proceeding as in 	\cite{CV}*{Proposition $4.5.3$}, we want to show that level sets of the minimal minimizer are Birkhoff sets.
\begin{prop}\label{Birksubsuplev}
Given $\theta \in \mathbb{R}$, the superlevel set $\{\uom>\theta \}$ has the Birkhoff property with respect to $\omega$, i.e.\
\begin{itemize}
\item $\{\tau_k \uom >\theta \}\subseteq \{\uom> \theta \}$, for any $k\in \mathbb{Z}^N$ such that $\omega \cdot k \le 0$, and
\item $\{\tau_k \uom >\theta \}\supseteq \{\uom> \theta \}$, for any $k\in \mathbb{Z}^N$ such that $\omega \cdot k \ge 0$. 
\end{itemize}
In the same way the sublevel set $ \{\uom< \theta \}$ has the Birkhoff property with respect to $-\omega$.

The Proposition still holds if strict levels are replaced by the broad ones.
\end{prop}

\begin{proof}
Denote with $v:=\min \{\uom, \tau_k \uom \}$ and note that $\tau_k \uom$ is the minimal minimizer with respect to $\tau_k S_\omega^{A,B}=S_\omega^{A+\omega \cdot k,B+\omega \cdot k}$. Now, if $\omega \cdot k \le 0$, from Lemma \ref{mininM} we have that $v\in M_\omega^{A+\omega \cdot k,B+\omega \cdot k}$, so that $\tau_k \uom \le v \le \uom$. Therefore
\[
\{\tau_k \uom >\theta \}\subseteq \{\uom>\theta \}.
\]
Similarly, if $\omega \cdot k \ge 0$ we get that $v\in \M$ and hence
\[
\{\uom >\theta \}\subseteq \{\tau_k \uom > \theta \}.
\]
For the conclusion concerning the sublevel set $\{\uom \le \theta \}$ and the superlevel set $\{\uom \ge ~\theta \}$ we can reason as in \cite{CV}*{Proposition $4.5.3$}.
\end{proof}

\subsection{Unconstrained and class $A$-minimization}
From now on we consider strips of the form
\[
S_\omega^M:=S_\omega^{0,M}=\{x\in \setR : \omega \cdot x \in [0,M]\}.
\]
We denote the space of admissible functions $\mathcal{A}_\omega^{0,M}$ with $\mathcal{A}_\omega^M$, the absolute minimizers with $\mathcal{M}_\omega^M$ and the minimal minimizer with $u_\omega^M$. Since we want to avoid narrow strips, we assume $M>10|\omega|$.

The goal of this subsection is to show that, for large universal values of $M/|\omega|$, the minimal minimizer $u_\omega^M$ becomes unconstrained, i.e.\ it no longer feels boundary data prescribed outside $S_\omega^M$, gaining additional minimizing properties in the whole $\setR$.
First of all we adapt the results of Section \ref{sec2} and Section \ref{sec3} to the minimal minimizer $u_\omega^M$ and in view of Corollary~ \ref{uomislocmin}, we have the existence of universal quantities $\alpha \in (0,1)$ and $C_1 \ge 1$ such that for any open $S\subset \subset S_\omega^M$ with $\dist(S, \partial S_\omega^M)\ge 1$ it holds
\begin{equation}\label{regminmin}
\|u_\omega^M\|_{C^{0,\alpha}(S)}\le C_1.
\end{equation}
Then from Proposition \ref{stimedensita}, fixed $x_0\in S_\omega^M$ and $R\ge 3$ such that $B_{R+2}(x_0)\subset \subset S_\omega^M$ we get that
\begin{equation}\label{stimeenergia}
\E(u_\omega^M,B_R(x_0))\le C_2 R^{N-1}\Psi_R(R)
\end{equation}
where $C_2>0$ is a universal constant and $\Psi_R(R)$ is defined in $\eqref{psidef}$. 

These two inequalities have a crucial role to show the main result of this section:
\begin{thm}\label{Thimp}
There exists a universal constant $M_0>0$ such that if $M\ge M_0|\omega|$, the distance between the superlevel set $\{u_\omega^M >-1+\delta_0 \}$ and the upper constraint $\{ \omega \cdot x=M \}$ delimiting $S_\omega^M$ is at least $1$.
\end{thm}

\begin{proof}
First of all we point out that during this proof we will denote balls $B$ and cubes $Q$ without expliciting the center. Then we claim that
\begin{equation}\label{claim}
\begin{aligned}
&\exists\; M_0\ge 8N\text{ universal constant}\; \text{ such that, for any}\; M\ge M_0|\omega|,\text{ we find a ball}\\
&B_{\sqrt{N}}(\bar{z})\subset \subset S_\omega^M\text{for some}\; \bar{z}\in S_\omega^M \text{ on which either}\; u_\omega^M\ge 1-\delta_0\text{ or}\; u_\omega^M \le-1+\delta_0.
\end{aligned}
\end{equation}
Given $M\ge 8N|\omega|$, assume that for every ball $\tilde{B}_{\sqrt{N}}\subset \subset S_\omega^M$ we can find a point $\tilde{x}\in \tilde{B}_{\sqrt{N}}$ such that $|u_\omega^M(\tilde{x})|<1-\delta_0$. If we prove that $M/|\omega| \le M_0$, then claim $\eqref{claim}$ follows.

Proceeding as in \cite{CV}*{Proposition $4.6.1$} we take $k\ge 2$ the only integer such that
\begin{equation}\label{K}
k\le \frac{M}{4N|\omega|}<k+1.
\end{equation}
Then we let $x_0 \in S_\omega^M$ be a point on the hyperplane $\{\omega \cdot x=\frac{M}{2} \}$ and $B=B_{nk}(x_0)$. Thanks to $\eqref{K}$ we get that $B \subset \subset S_\omega^M$ with
\begin{equation}\label{distpos}
\dist (B, \partial S_\omega^M)=\frac{M}{2|\omega|}-Nk\ge Nk\ge 4.
\end{equation}
Therefore we can apply $\eqref{regminmin}$ to obtain
\begin{equation}\label{primaclaim}
\|u_\omega^M\|_{C^{0,\alpha}(B)}\le C_1.
\end{equation}
Now we consider $Q$ a cube with center in $x_0$ and side $2\sqrt{N}k$. Clearly $Q\subset B$ and we can partition it (up a negligible set) into a collection $\{Q_j\}_{j=1}^{k^N}$ of cubes with sides $2\sqrt{N}$ parallel to those of $Q$. Then we call $B_j \subset Q_j$ the ball of radius $\sqrt{N}$ with the same center of $Q_j$.

 By our initial assumption, for every $j=1,\cdots, k^N$, there exists $\tilde{x}_j \in B_j$ such that $|u_\omega^M(\tilde{x}_j)|< 1-\delta_0$.

We claim that 
\begin{equation}\label{abbassodelta0}
|u_\omega^M|<1-\delta_1\quad \text{in}\; B_{r_0}(\tilde{x}_j)
\end{equation}
for some $\delta_1<\delta_0$ and some universal radius $r_0\in (0,1)$. Indeed, defining $r_0:=\Big( \frac{\delta_0-\delta_1}{C_1}\Big)^{1/\alpha}$, by $\eqref{primaclaim}$, we have
\[
|u_\omega^M|\le |u_\omega^M (\tilde{x}_j)|+C_1|x-\tilde{x}_j|^\alpha<1-\delta_0+C_1r_0^\alpha=1-\delta_1,
\]
for any $x\in B_{r_0}(\tilde{x}_j)$ and the claim is proved. 

On the other hand, since $\tilde{x}_j \in B_j\subset Q_j$, we get
\begin{equation}\label{secclaim}
|B_{r_0}(\tilde{x}_j)\cap Q_j|\ge \frac{1}{2^N}|B_{r_0}(\tilde{x}_j)|=\frac{\alpha_N}{2^N}r_0^N.
\end{equation}
Therefore, from $\eqref{abbassodelta0}$, $\eqref{secclaim}$ and $\eqref{W2}$ we obtain
\begin{equation}
\begin{aligned}
\mathscr{P}(u_\omega^M,B)&\ge \mathscr{P}(u_\omega^M,Q)=\sum_{j=1}^{k^N}{\mathscr{P}(u_\omega^M,Q_j)}\ge \sum_{j=1}^{k^N}{\mathscr{P}(u_\omega^M,B_{r_0}(\tilde{x}_j)\cap Q_j})\\
&=\sum_{j=1}^{k^N}{\int_{B_{r_0}(\tilde{x}_j)\cap Q_j}\Big( W(x,u_\omega^M(x))+H(x)u_\omega^M(x)\Big)\di x}\\
&\ge \Big[ \gamma (1-\delta_1)-\eta(\delta_1-1)\Big]\sum_{j=1}^{k^N}{|B_{r_0}(\tilde{x}_j)\cap Q_j|} \\
& \ge \Big[ \gamma (1-\delta_1)-\eta(\delta_1-1)\Big]\frac{\alpha_N}{2^N}r_0^N k^N:=C_3k^N,
\end{aligned}
\end{equation}
where $C_3>0$ is a universal constant.
Moreover from $\eqref{stimeenergia}$ (that we can apply to $B$ thanks to $\eqref{distpos}$)
\[
\mathscr{P}(u_\omega^M,B)\le \E(u_\omega^M,B)\le C_2(Nk)^{N-1}\Psi_s(Nk)\le C_4 k^{N-1}\Psi_s(k),
\]
for some universal $C_4>0$. By $\eqref{K}$, the same holds for $M/|\omega|$, so $\eqref{claim}$ follows.

Thus it remains to show that $u_\omega^M$ cannot be greater or equal to $1-\delta_0$ on $B_{\sqrt{N}}(\bar{z})$.

Assume by contradiction that
\begin{equation}\label{hpass}
u_\omega^M\ge 1-\delta_0\quad \text{in}\; B_{\sqrt{N}}(\bar{z}).
\end{equation}
Using Proposition \ref{Birksubsuplev} we have that the set $\{u_\omega^M\ge 1-\delta_0\}$ has the Birkhoff property with respect to $\omega$. Then, $\eqref{hpass}$ and Proposition \ref{Birkricito} imply that the superlevel set contains the half-space $\Pi_-:=\{\omega \cdot (x-\bar{z})<0 \}$. Since $B_{\sqrt{N}}(\bar{z})\subset S_\omega^M$, we can affirm that $\partial \Pi_-$ is at least at distance $1$ from the level constraint $\{\omega \cdot x=0\}$. As a consequence if we suppose w.l.o.g.\ that $\omega_1>0$, the translation $\tau_{-e_1}u_\omega^M \in \mathcal{A}_\omega^M$. In view of the periodicity assumptions $\eqref{K3}$, $\eqref{W4}$ and $\eqref{Hperiodica}$, we get that $\mathcal{F}_\omega(\tau_{-e_1}u_\omega^M)=\mathcal{F}_\omega(u_\omega^M)$ and so $\tau_{-e_1}u_\omega^M\in \mathcal{M}_\omega^M$. Then, since $u_\omega^M$ is the minimal minimizer, it results
\[
u_\omega^M(x+e_1)=\tau_{-e_1}u_\omega^M\ge u_\omega^M(x)\quad \text{for a.e.}\; x\in \setR.
\]
Now, iterating this inequality we obtain
\[
u_\omega^M(x+te_1)\ge u_\omega^M(x)\ge1-\delta_0\quad \text{for a.e.}\; x\in \Pi_-\;\text{and}\; t\in \N
\]
or equivalently $u_\omega^M\ge 1-\delta_0$ a.e.\ in $\setR$ that contradicts the fact that $u_\omega^M\le-1+\delta_0$ in $\{ \omega \cdot x\ge M\}$ by construction. As a consequence $u_\omega^M\le-1+\delta_0$ on the ball $B_{\sqrt{N}}(\bar{z})$, hence applying again Proposition \ref{Birksubsuplev} and Proposition \ref{Birkricito} to the sublevel set $\{u_\omega^M \le-1+\delta_0 \}$, we prove the theorem.
\end{proof}

\begin{coro}\label{coroupper}
If $M\ge M_0|\omega|$, then $u_\omega^M=u_\omega^{M+a}$ for all $a\ge 0$.
\end{coro}

\begin{proof}
Given $M\ge M_0|\omega|$ and $a\in [0,1]$, we may apply Theorem \ref{Thimp} to the minimal minimizer $u_\omega^{M+a}$ to obtain that $u_\omega^{M+a}\le-1+\delta_0$ a.e.\ in the half-space $\{\omega \cdot x\ge M \}$. But then $u_\omega^{M+a}\in \mathcal{A}_\omega^M$ and by minimality of $u_\omega^M$, we get that $\mathcal{F}_\omega(u_\omega^M)\le \mathcal{F}_\omega(u_\omega^{M+a})$.

On the other hand, obviously $u_\omega^M \in \mathcal{A}_\omega^{M+a}$, so $\mathcal{F}_\omega(u_\omega^{M+a})\le  \mathcal{F}_\omega(u_\omega^M)$. Therefore $u_\omega^M$ and $u_\omega^{M+a}$ belong to $\mathcal{M}_\omega^M \cap \mathcal{M}_\omega^{M+a}$ and hence they are the same function. 
Iterating this argument we can extend this result to any $a\ge 0$.
\end{proof}

Roughly speaking, this corollary tells us that if $M/|\omega|$ is greater than the universal constant $M_0$ found in Theorem \ref{Thimp}, the upper constraint $\{\omega \cdot x=M \}$ becomes irrelevant for the minimal minimizer $u_\omega^M$ which achieves values below $-1+\delta_0$ before touching the constraint.

In the next proposition we show that we have an analogous behaviour with the lower constraint $\{\omega \cdot x=0 \}$ and hence we get that the minimal minimizer $u_\omega^M$ is unconstrained.

\begin{prop}\label{minunconstrained}
If $M\ge M_0 |\omega|$, then $u_\omega^M \in \mathcal{M}^{-a,M+a}$ for any $a\ge 0$, i.e.\ $u_\omega^M$ is unconstrained.
\end{prop}

\begin{proof}
Fix $k\in \mathbb{Z}^N$ such that $\omega \cdot k\ge a$. Let $v\in \mathcal{A}_\omega^{-a,M+a}$ and consider its translation $\tau_k v\in \mathcal{A}_\omega^{M+a+\omega \cdot k}$. Corollary \ref{coroupper} tells us that $\mathcal{F}_\omega(u_\omega^M)\le \mathcal{F}_\omega(\tau_k v)$ and the thesis follows since by $\eqref{K3}$, $\eqref{W4}$ and $\eqref{Hperiodica}$, we have that $\mathcal{F}_\omega (v)=\mathcal{F}_\omega(\tau_k v)$.
\end{proof}

We conclude this subsection combining the previous proposition with the results of Subsection \ref{subsec4.1} obtaining that $u_\omega^M$ is a class $A$ minimizer.

\begin{thm}
If $M\ge M_0|\omega|$, then $u_\omega^M$ is a class $A$ minimizer of the functional $\E$.
\end{thm}

\begin{proof}
Let $\Omega \subseteq \setR$ be a bounded subset. Take $a\ge 0$ and $m\in \mathbb{Z}^{N-1}$ such that $\Omega \subset \subset ~\tilde{S}_{\omega,m}^{-a,M+a}$. From Proposition \ref{minimiuguali}, it follows that $u_\omega^{-a,M+a}$ is the minimal minimizer of the class $\mathcal{M}_{\omega,m}^{-a,M+a}$. On the other hand Proposition \ref{minunconstrained} implies that $\mathcal{F}_\omega(u_\omega^M)=\mathcal{F}_\omega(u_\omega^{-a,M+a})$. Then
\[
\mathcal{F}_{\omega, m}(u_\omega^M)=c_m \mathcal{F}_\omega (u_\omega^M)=c_m\mathcal{F}_\omega(u_\omega^{-a,M+a})=\mathcal{F}_{\omega,m}(u_\omega^{-a,M+a}),
\]
where $c_m=\Pi_{i=1}^{N-1}{m_i}$.

Therefore $u_\omega^M \in \mathcal{M}_{\omega, m}^{-a,M+a}$ and Proposition \ref{minimizerisminE} yields that $u_\omega^M$ is a local minimizer of $\E$ in $\Omega$.
\end{proof}

\subsection{The case of irrational directions}\label{subsect4.7}

In this subsection we want to prove Theorem~\ref{Mainth} with the assumption $\eqref{K4}$, also for irrational vectors $\omega$. We will use an approximation argument as in \cite{CV}*{Subsection $4.7$}.

Taken $\omega \in \setR \setminus \mathbb{Q}^N$, we consider a sequence $\{\omega_j\}_{j\in \N} \subset \mathbb{Q}^N \setminus \{ 0\}$ such that $\omega_j \rightarrow \omega$. Denoting with $u_j$ the class $A$ minimizer given by our construction which corresponds to $\omega_j$ we know that $u_j \in H_{\text{loc}}^s(\setR) \cap L^\infty(\setR)$ with $|u_j|\le 1+\delta_0$ in $\setR$ and
\begin{equation}\label{irr1}
\Big \{x \in \setR : |u_j(x)|\le  1-\delta_0 \Big \} \subseteq  \Big \{ x \in \setR : \frac{\omega_j}{|\omega_j|}\cdot x \in [0,M_0] \Big \}
\end{equation}
for any $j\in \N$. Moreover Theorem \ref{holderianita} implies that the $u_j$'s are uniformly bounded in $C^{0,\alpha}(\setR)$ for some unversal $\alpha \in (0,1)$. So, thanks to Ascoli-Arzelà Theorem we can find a subsequence of $\{ u_j\}_{j\in \N}$ (not relabeled) converging to some continuous function $u$, uniformly on compact subsets of $\setR$ and $|u|\le 1+\delta_0$ in $\setR$. Since condition $\eqref{irr1}$ passes to the limit, the same inclusion holds if we replace $u_j$ and $w_j$ with $u$ and $w$. Hence, to prove Theorem \ref{Mainth} we only need to check that $u$ is a class $A$ minimizer of $\E$. With this aim in mind we fix $R\ge 1$ and we claim that $u$ is a local minimizer of $\E$ in $B_R$, i.e.\ $\E(u, B_R)<+\infty$ and
\begin{equation}\label{irr2}
\E(u,B_R)\le \E(u+\varphi, B_R)\quad \text{for any}\; \varphi \; \text{such that }\;\text{spt }\varphi \subset B_R.
\end{equation}
Thanks to Remark \ref{minimisottoins}, this will implies that $u$ is a class $A$ minimizer.

To show $\eqref{irr2}$ we apply Theorem \ref{stimedensita} to $u_j$ so
\[
\E(u_j,B_{R+1})\le C_R
\]
for some constant $C_R>0$ independent of $j$. Moreover by an application of Fatou's Lemma
\begin{equation}\label{irr3}
\E(u,B_{R+\tau})\le \liminf_{j\rightarrow +\infty}\E(u_j,B_{R+\tau})
\end{equation}
for any $\tau \in [0,1]$. In particular
\begin{equation}\label{irr4}
\E(u,B_R)\le \E(u,B_{R+1})\le C_R<+\infty
\end{equation}
because $\E(u, \cdot)$ is monotone non-decreasing with respect to set inclusion.

As it concerns the right hand side of $\eqref{irr3}$ we let $\{\eps_j\}_{j\in \N}$ such that
\[
\eps_j:=\|u_j-u\|_{L^\infty(B_{R+1})}.
\]
It is easy to see that $\eps_j\rightarrow 0$ and we may suppose that $\eps_j \le 1/2$ for any $j\in \N$. Then we consider $\eta_j \in C_c^\infty(\setR)$ a cut-off function such that $0\le \eta_j\le 1$ in $\setR$, $\eta_j=1$ in $B_R$, spt $(\eta_j)\subseteq B_{R+\eps_j}$ and $|\nabla \eta_j|\le 2/\eps_j$ in $\setR$. Take $\varphi$ as in $\eqref{irr2}$ and assume w.l.o.g.\ that $\varphi \in L^\infty(\setR)$. We also suppose that $\E(u+\varphi,B_R)<+\infty$, otherwise $\eqref{irr2}$ is obviously satisfied. Consequently, using $\eqref{irr4}$, $\eqref{K2}$ and the boundedness of $H$, $u$ and $\varphi$ we get that $\varphi \in H^s(B_{R+1})$. At this point we define $v:=u+\varphi$ and
\[
v_j:=\eta_j u+(1-\eta_j)u_j+\varphi\quad \text{in}\; \setR.
\]
Observe that $v_j=v$ in $B_R$ and $v_j=u$ in $\setR \setminus B_{R+\eps_j}$, hence $v_j$ is an admissible competitor for $u_j$ in $B_{R+\eps_j}$. Then, being $u_j$ minimizer,
\[
\E(u_j,B_{R+\eps_j})\le \E (v_j, B_{R+\eps_j}).
\]
Moreover $v_j\rightarrow v$ uniformly on compact subsets of $\setR$ and
\[
\|v_j-v\|_{L^\infty(B_{R+1})}\le \|u_j-u\|_{L^\infty(B_{R+1})}=\eps_j.
\]
Now, we can proceed as in \cite{CV}*{Pag. $32-34$} and then we can use
\[
\mathscr{P}(v_j,B_{R+\eps_j})\le \mathscr{P}(v,B_R)+W^*|B_{R+\eps_j} \setminus B_R|+\eta \|u\|_{L^\infty(\setR)}|B_{R+\eps_j} \setminus B_R|
\]
to say that there exists a function $r:(0,1)\rightarrow (0,+\infty)$ for which
\begin{equation}\label{delta0}
\lim_{\delta \rightarrow 0^+}r(\delta)=0
\end{equation}
such that
\[
\limsup_{j\rightarrow +\infty}\E(u_j,B_R)\le \E(v,B_R)+r(\delta).
\]
Combining this inequality with $\eqref{irr3}$ we have
\[
\E(u,B_R)\le \E(v,B_R)+r(\delta)
\]
and since $\delta$ is arbitrary, from $\eqref{delta0}$, we obtain $\eqref{irr2}$ and hence that $u$ is a class $A$ minimizer of $\E$.

\section{Proof of Theorem \ref{Mainth} for general kernels}\label{sec5}
In this section we want to prove Theorem \ref{Mainth} also
for kernel not satisfying condition $\eqref{K4}$. Indeed none of the estimates that we showed there involve any of the parameters appearing in $\eqref{K4}$. So we can use a limit argument similar to this of Section \ref{subsect4.7}.

Let $K$ be a kernel satisfying $\eqref{K1}$, $\eqref{K2}$, $\eqref{K3}$ and consider a monotone increasing sequence $\{R_j\}_{j\in \N}\subset [2,+\infty)$ diverging to $+\infty$. We define
\[
K_j(x,y):=K(x,y)\chi_{[0,R_j]}(|x-y|)\quad \text{for any}\;x,y\in \setR
\]
and we observe that it fulfills $\eqref{K1}$, $\eqref{K2}$, $\eqref{K3}$. Moreover $K_j$ satisfies $\eqref{K4}$ with $\bar{R}=R_j$. Call $\E_j$ the energy functional $\eqref{FinOmega}$ corresponding to $K_j$ and, fixed a direction $\omega \in \setR \setminus \{ 0\}$, we denote with $u_j$ the plane-like class $A$-minimizer for $\E_j$ with direction $\omega$. Since $K_j$ verifies $\eqref{K4}$ these minimizers exist thanks to Section \ref{sec4}. We have
\begin{equation}\label{constraintuj}
\Big \{ x\in \setR :|u_j(x)|\le 1-\delta_0\Big \} \subseteq \Big \{x\in \setR :\frac{\omega}{|\omega|}\cdot x\in [0,M_0]\Big \},
\end{equation}
for a universal value $M_0>0$. We also know that $|u_j|\le 1+\delta_0$ in $\setR$ and, thanks to Theorem \ref{holderianita}, $\|u_j\|_{C^{0,\alpha}(\setR)}\le C$ for some $\alpha \in (0,1]$. We underline that, since $K_j$ satisfies $\eqref{K2}$ with the same structural constants, we can choose $M_0$, $\alpha$ and $C$ independent of $j$. As a consequence, Ascoli-Arzelà Theorem implies that, up to a subsequence, $\{u_j\}$ converges to a continuous function $u$, uniformly on compact subsets of $\setR$. The limit function $u$ satisfies $\eqref{constraintuj}$ and, if $\omega$ is rational, each $u_j$ is $\sim$-periodic, hence $u$ is $\sim$-periodic. To show that $u$ is a class $A$ minimizer, we fix $R\ge 1$ and we take a perturbation $\varphi$ with spt $\varphi\subset \subset B_R$. We know that
\[
\E_j(u_j,B_R)\le \E_j(u_j+\varphi,B_R)\quad \text{for any}\; j\in \N.
\]
On the other hand from an application of Fatou's Lemma we get
\[
\E(u,B_R)\le \liminf_{j\rightarrow +\infty}\E_j(u_j,B_R)
\]
and following the reasoning of the Subsection \ref{subsect4.7} we have that
\[
\limsup_{j\rightarrow +\infty}\E_j(u_j,B_R)\le \E(u+\varphi,B_R).
\]
These two inequalities tell us that $u$ is a class $A$ minimizer of $\E$ so Theorem \ref{Mainth} is completely proved.

\section{Acknowledgements}\label{sec6}
The author is supported by the PRIN project ``Variational methods, with applications to problems in mathematical physics and geometry''. 
The author thanks A. Malchiodi and M. Novaga for the fruitful discussions and their kind support.

\end{document}